\documentclass[11pt]{amsart}
\usepackage{amsmath,amsthm,latexsym,amsfonts,amssymb}
\usepackage{datetime}
\usepackage{graphicx}

\newtheoremstyle{slthm}
{9pt}
{5pt}
{\slshape}
{}
{\bfseries}
{.}
{.5em}
{\thmname{#1}\thmnumber{ #2}\thmnote{ (#3)}}

\newtheoremstyle{prcl}
{9pt}
{5pt}
{\slshape}
{}
{\bfseries}
{.}
{.5em}
{\thmname{#3}\thmnumber{ #2}}

\newtheoremstyle{prblm}
{9pt}
{5pt}
{\rm}
{}
{\bfseries}
{.}
{.5em}
{\thmname{#3}\thmnumber{ #2}}

\theoremstyle{slthm}
\newtheorem{thm}{Theorem}[section]
\newtheorem{lemma}[thm]{Lemma}
\newtheorem{prop}[thm]{Proposition}
\newtheorem{cor}[thm]{Corollary}
\newtheorem{fact}[thm]{Fact}

\theoremstyle{definition}

\newtheorem{df}[thm]{Definition}

\newtheorem{rmkdf}[thm]{Remark and Definition}
\newtheorem{nrmk}[thm]{Remark}
\newtheorem{nrmks}[thm]{Remarks}
\newtheorem{expl}[thm]{Example}
\newtheorem{expls}[thm]{Examples}

\theoremstyle{remark}
\newtheorem*{rmk}{Remark}

\theoremstyle{prcl}

\theoremstyle{prblm}

\newenvironment{renumerate}
        {\renewcommand{\labelenumi}{\rm (\roman{enumi})}
         \begin{enumerate}}
        {\end{enumerate}}

\newcounter{flexnummark}

\DeclareMathOperator{\im}{Im}
\DeclareMathOperator{\re}{Re}
\DeclareMathOperator{\supp}{supp}
\DeclareMathOperator{\ord}{ord}

\DeclareMathOperator{\geom}{G_{\text{eom}}}
\DeclareMathOperator{\bexp}{\textbf{exp}}
\DeclareMathOperator{\blog}{\textbf{log}}

\newcommand{\into}{\longrightarrow}

\renewcommand{\bar}{\overline}

\def\Ind#1#2{#1\setbox0=\hbox{$#1x$}\kern\wd0\hbox to 0pt{\hss$#1\mid$\hss}
\lower.9\ht0\hbox to 0pt{\hss$#1\smile$\hss}\kern\wd0}

\def\Notind#1#2{#1\setbox0=\hbox{$#1x$}\kern\wd0\hbox to 0pt{\mathchardef
\nn=12854\hss$#1\nn$\kern1.4\wd0\hss}\hbox to
0pt{\hss$#1\mid$\hss}\lower.9\ht0 \hbox to
0pt{\hss$#1\smile$\hss}\kern\wd0}

\newcommand{\set}[1]{\left\{#1\right\}}

\newcommand{\NN}{\mathbb{N}}
\newcommand{\ZZ}{\mathbb{Z}}

\newcommand{\RR}{\mathbb{R}}

\newcommand{\CC}{\mathbb{C}}
\newcommand{\TT}{\mathbb{T}}

\newcommand{\curly}[1]{\mathcal{#1}}
\newcommand{\A}{\curly{A}}

\newcommand{\C}{\curly{C}}
\newcommand{\D}{\curly{D}}

\newcommand{\F}{\curly{F}}

\newcommand{\K}{\curly{K}}

\renewcommand{\aa}{\mathbf{a}}
\newcommand{\bb}{\mathbf{b}}
\newcommand{\cc}{\mathbf{c}}
\newcommand{\ee}{\mathbf{e}}
\newcommand{\ff}{\mathbf{f}}
\renewcommand{\gg}{\mathbf{g}}
\newcommand{\hh}{\mathbf h}

\renewcommand{\ll}{\mathbf l}
\newcommand{\mm}{\mathbf{m}}
\newcommand{\nn}{\mathbf{n}}
\newcommand{\pp}{\mathbf{p}}
\newcommand{\qq}{\mathbf{q}}

\newcommand{\xx}{\mathbf{x}}

\DeclareMathOperator{\Ranexp}{\RR_{an,exp}}

\newcommand{\Ps}[2]{\mathbb{#1}[\![#2]\!]}

\newcommand{\As}[2]{#1[\![#2]\!]}

\newcommand{\Gs}[2]{#1(\!(#2)\!)}

\numberwithin{equation}{section}

\allowdisplaybreaks[2]

\title {Quasianalytic Ilyashenko algebras} 

\author {Patrick Speissegger}

\address {Department of Mathematics and Statistics, McMaster University, 1280
Main Street West, Hamilton, Ontario L8S 4K1, Canada}

\email {speisseg@math.mcmaster.ca}

\date{\today\ at \currenttime}



\thanks{Supported by NSERC Discovery Grant \#261961.  I thank Zeinab Galal, Tobias Kaiser,  Jean-Philippe Rolin and Tamara Servi for valuable discussions on these notes, and the referee for their careful reading and valuable feedback.}

\begin{document}

\begin{abstract}
I construct a quasianalytic field $\F$ of germs at $+\infty$ of real functions with logarithmic generalized power series as asymp\-totic expansions, such that $\F$ is closed under differentiation and $\log$-composition; in particular, $\F$ is a Hardy field.  Moreover, the field $\F \circ (-\log)$ of germs at $0^+$ contains all transition maps of hyperbolic saddles of planar real analytic vector fields.
\end{abstract}

\maketitle
\markboth{}{}

\section{Introduction}

In his solution of Dulac's problem, Ilyashenko \cite{Ilyashenko:1991fk} introduces the class $\A$ of germs at $+\infty$ of \textit{almost regular} functions, and he shows that this class is quasianalytic and \textbf{closed under $\log$-composition}, by which I mean the following: given $f,g \in \A$ such that $\lim_{x \to +\infty} \frac1g(x) = +\infty$, it follows that $f \circ (-\log) \circ g \in \A$.  As a consequence, $\A \circ (-\log)$ is a quasianalytic class of germs at $0^+$ that is closed under composition.  Ilyashenko also shows that if $f$ is the germ at $0^+$ of a transition map near a hyperbolic saddle of a planar real analytic vector field, then $f$ belongs to $\A \circ (-\log)$; from this, it follows that limit cycles of a planar real analytic vector field $\xi$ do not accumulate on a hyperbolic polycycle of $\xi$.  (For a discussion of Dulac's problem and related terminology used here, we refer the reader to Ilyashenko and Yakovenko \cite[Section 24]{Ilyashenko:2008fk}.  The class $\A$ also plays a role in the description of Riemann maps and solutions of Dirichlet's problem on semianalytic domains; see Kaiser \cite{MR2495077,MR2481954} for details.)  

That $\A$ is closed under $\log$-composition is due to a rather peculiar assumption built into the definition of ``almost regular'': by definition, a function $f:(a,+\infty) \into \RR$ is \textbf{almost regular} if there exist real numbers $0 \le \nu_0 < \nu_1 < \dots$ such that $\lim_i \nu_i = +\infty$, polynomials $p_i \in \RR[X]$ for each $i$ and a \textbf{standard quadratic domain} $$\Omega = \Omega_C := \set{z+C \sqrt{1+z}:\ \re z > 0} \subseteq \CC,\quad\text{with } C>0,$$ such that 
\begin{renumerate}
	\item $f$ has a bounded holomorphic extension $\ff:\Omega \into \CC$;
	\item $p_0$ is a nonzero constant and, for each $N \in \NN$, $$\ff(z) - \sum_{i=0}^N p_i(z) e^{-\nu_i z} = o\left(e^{-\nu_N z}\right)\quad\text{as } |z| \to +\infty \text{ in } \Omega.$$
\end{renumerate}

\begin{rmk}
	 For an almost regular $f$ as defined here, the function $\log \circ f$ is almost regular in the sense of \cite[Definition 24.27]{Ilyashenko:2008fk}.  
\end{rmk}

It is the assumption that $p_0$ be a nonzero constant that makes the class $\A$ closed under $\log$-composition.  However, one drawback of this assumption is that the class $\A$ is not closed under addition (because of possible cancellation of the leading terms), which makes it unamenable to study by many commonly used algebraic-geometric methods.

I show here that Ilyashenko's construction of $\A$ can be adapted, using his notion of \textit{superexact asymptotic expansion} \cite[Section 0.5]{Ilyashenko:1991fk}, to obtain a quasianalytic class $\F$ that is closed under addition and multiplication, contains $\exp$ and $\log$ and is closed under differentiation and $\log$-composition.  This construction comes at the cost of replacing the asymptotic expansions above by the following series: for $k \in \ZZ$, we denote by $\log_k$ the $k$th compositional iterate of $\log$.  Recall from van den Dries and Speissegger \cite{Dries:1998xr} that a \textbf{generalized power series} is a power series $F = \sum_{\alpha \in \RR^k} a_\alpha X^\alpha$, where $X = (X_1, \dots, X_k)$, each $a_\alpha \in \RR$ and the \textbf{support} of $F$,
\[ \supp(F):= \set{\alpha \in \RR^n:\ a_\alpha \ne 0}, \] is contained in a cartesian product of well-ordered subsets of $\RR$; the set of all generalized power series in $X$ is denoted by $\Ps{R}{X^*}$.  Moreover, we call the support of $F$ \textbf{natural} (see Kaiser et al. \cite{Kaiser:2009ud}) if, for every compact box $B \subseteq \RR^k$, the intersection of $B \cap \supp(F)$ is finite.

\begin{df}
	\label{log-power_series}
	A \textbf{logarithmic generalized power series} is a series of the form $F\left(\frac1{\log_{i_1}}, \dots, \frac1{\log_{i_k}}\right)$, where $i_1, \dots, i_k \ge -1$ and $F \in \Ps{R}{X^*}$ has natural support.
\end{df}
	
I denote by $L$ the divisible multiplicative group of all monomials of the form $(\log_{i_1})^{r_1} \cdots (\log_{i_k})^{r_k}$, with $-1 \le i_1 < \cdots < i_k$ and $r_1, \dots, r_k \in \RR$.  Note that $L$ is linearly ordered by setting $m \le n$ if and only if $\lim_{x \to +\infty} \frac{m(x)}{n(x)} \le 1$.  (In fact, $L$ is a multiplicative subgroup of the Hardy field of all germs at $+\infty$ of functions definable in the o-minimal structure $\RR_{\exp}$, see Wilkie \cite{Wilkie:1996ce}.)  Indeed, this ordering can be described as follows:  identify each $m \in L$ with a function $m:\{-1\} \cup \NN \into \RR$ in the obvious way.  Then for $m,n \in L$ we have $m<n$ in $L$ if and only if $m<n$ in $\RR^{\{-1\} \cup \NN}$ in the lexicographic ordering.

For a divisible subgroup $L'$ of $L$, I denote by $\Ps{R}{L'}$ the set of all logarithmic generalized power series with support contained in $L'$.  Note that, by definition, every series in $\Ps{R}{L'}$ has support contained in $$L ' \cap \set{(\log_{i_1})^{r_1} \cdots (\log_{i_k})^{r_k}:\ -1 \le i_1 < \cdots < i_k \text{ and } r_1, \dots, r_k \le 0}.$$  It is straightforward to see that $\Ps{R}{L'}$ is an $\RR$-algebra under the usual addition and multiplication of series, and I denote its fraction field by $\Gs{\RR}{L'}$.  (So the general series in $\Gs{\RR}{L'}$ is of the form $m F$, where $m \in L'$ and $F \in \Ps{R}{L'}$.)  This notation agrees with the usual notation for \textit{generalized series}, see for instance van den Dries et al \cite{Dries:2001kh}.  To simplify notations, I sometimes write $F \in \Gs{\RR}{L'}$ as $F = \sum_{m \in L'} a_m m$ as in \cite{Dries:2001kh}; in this situation, I call the set$$\supp(F) := \set{m \in L':\ a_m \ne 0}$$ the \textbf{support} of $F$.  Note that, under the ordering on $L'$, the set $\supp(F)$ is a reverse well-ordered subset of $L'$ of order-type at most $\omega^k$ for some $k \in \NN$.  I call $\supp(F)$ \textbf{$L'$-natural} if $\supp(F) \cap (m,+\infty)$ is finite for any $m \in L'$.

For $F = \sum_{m \in L'} a_m m \in \Gs{\RR}{L'}$ and $n \in L$, I denote by $$F_n:= \sum_{m \ge n} a_m m$$ the \textbf{truncation} of $F$ above $n$.  A subset $A \subseteq \Gs{\RR}{L'}$ is \textbf{truncation closed} if, for every $F \in A$ and $n \in L$, the truncation $F_n$ belongs to $A$.

Since the support of a logarithmic generalized power series can have order type $\omega^k$ for arbitrary $k \in \NN$, I need to make sense of what it means to have such a series as asymptotic expansion.  I do this in the context of an algebra of functions:

\begin{df}
  \label{relative_asymptotic}
  Let $\K$ be an $\RR$-algebra of germs at $+\infty$ of functions $f:(a,\infty) \into \RR$ (with $a$ depending on $f$), let $L'$ be a divisible subgroup of $L$, and let $T:\K \into \Gs{\RR}{L'}$ be an $\RR$-algebra homomorphism.  The triple $(\K,L',T)$ is a \textbf{quasianalytic asymptotic algebra} (or \textbf{qaa algebra} for short) if 
  \begin{renumerate}
    \item $T$ is injective;
    \item the image $T(\K)$ is truncation closed;
    \item for every $f \in \K$ and every $n \in L'$, we have
      \begin{equation*}
        f - T^{-1}((Tf)_n) = o(n).
      \end{equation*}
  \end{renumerate}
  In this situation, for $f \in \K$, I call $T(f)$ the \textbf{$\K$-asymptotic expansion} of $f$.
\end{df}

The result of this note can now be stated:

\begin{thm}
	\label{main_thm}
	\begin{enumerate}
	\item There exists a qaa field $(\F,L,T)$ that contains the class $\A$ as well as $\exp$ and $\log$.
	\item The field $\F$ is closed under differentiation and $\log$-composition.
	\end{enumerate}
\end{thm}

The remainder of this paper is divided into six sections: Section \ref{sqd} discusses some basic properties of standard quadratic domains, Section \ref{asymptotic_section} introduces strong asymptotic expansions, Section \ref{construction_section} contains the construction of $(\F,L,T)$, Section \ref{diff_section} contains the proof of closure under differentiation and Section \ref{comp_section} that of closure under $\log$-composition.  Finally, Section \ref{conclusion_section} contains some remarks putting this paper in a wider context.  

In Section \ref{comp_section}, I rely on the observation that $\Gs{\RR}{L}$ is a subset of the set $\TT$ of transseries as defined by van der Hoeven in \cite{Hoeven:2006qr}; I use, in particular, the fact that $\TT$ is a group under composition.

The construction of $\F$ is based on the following consequence of the Phrag\-m\'en-Lindel\"of principle \cite[Theorem 24.36]{Ilyashenko:2008fk}: 

\begin{fact}[{\cite[Lemma 24.37]{Ilyashenko:2008fk}}]
	\label{ph-l}
	Let $\Omega \subseteq \CC$ be a standard quadratic domain and $\phi:\Omega \into \CC$ be holomorphic.  
If $\phi$ is bounded and, for each $n \in \NN$,
		\begin{equation*}
		|\phi(x)| = o\left(e^{-nx}\right) \quad\text{as}\quad x \to +\infty \text{ in } \RR,
		\end{equation*}
		then $\phi = 0$.
\end{fact}

Indeed, I use this consequence of the Phragm\'en-Lindel\"of principle as a black box.  I suspect that other Phragm\'en-Lindel\"of principles, such as the one found in Borichev and Volberg \cite[Theorem 2.3]{MR1353489}, might be used in a similar way to obtain other qaa algebras.

\section{Standard quadratic domains}\label{sqd}

This section summarizes some elementary properties of standard quadratic domains and makes some related conventions.  For $a \in \RR$, I set $$H(a):= \set{z \in \CC:\ \re z > a},$$ and I define $\phi_C:H(-1) \into H(-1)$ by $$\phi_C(z):= z + C\sqrt{1+z}.$$  

\begin{figure}[htb] 
\label{sqd_figure}
\begin{center}
\includegraphics[scale = 0.3]{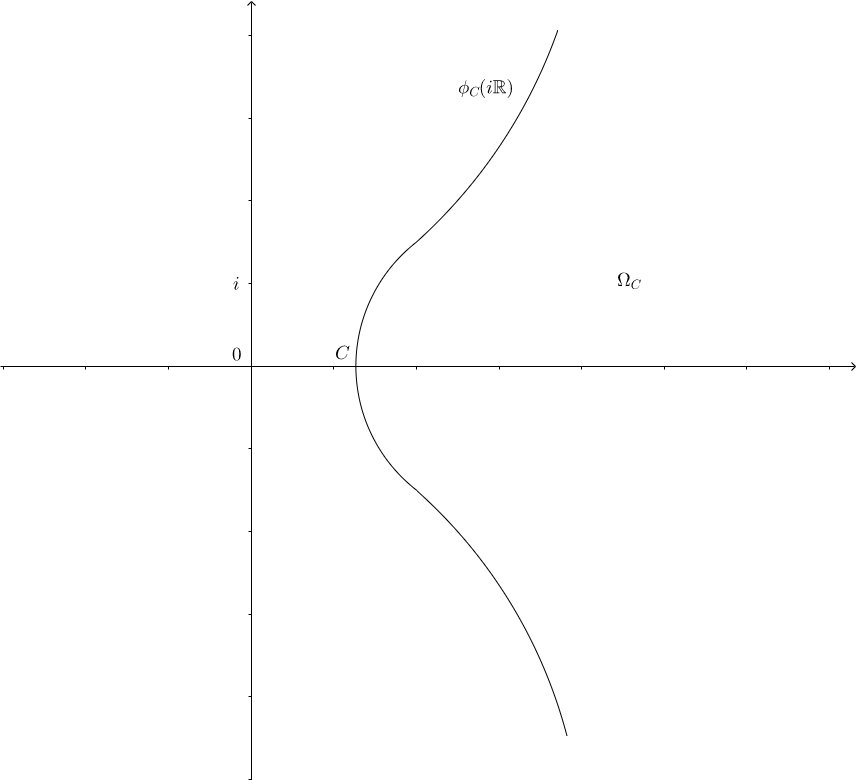} \caption{A standard quadratic domain and its boundary $\phi_C(i\RR)$} 
\end{center}
\end{figure}

I denote by $\C$ the set of all germs at $+\infty$ of continuous functions $f:\RR \into \RR$.  For $f,g \in \C$ I write $f \sim g$ if $f(x)/g(x) \to 1$ as $x \to +\infty$.

\begin{lemma}
	\label{sqd_lemma}
	Let $C>0$.
	\begin{enumerate}
		\item The map $\phi_C$ is conformal with compositional inverse $\phi_C^{-1}$ given by $$\phi_C^{-1}(z) = z + \frac{C^2}2 -C \sqrt{1+z+\frac{C^2}4};$$ in particular, the boundary of $\Omega_C$ is the set $\phi_C(i\RR)$.
		\item We have $\re\phi_C(ix) \sim \frac{C}{\sqrt 2}\sqrt{x}$ and $\im\phi_C(ix) \sim x$.
		\item There exists a continuous $f_C:[C,+\infty) \into (0,+\infty)$ such that $\im\phi_C(ix) = f_C(\re\phi_C(ix))$ for $x>0$ and  $f_C(x) \sim 2(x/C)^2$.
	\end{enumerate}
\end{lemma}

\begin{proof}
	These observations are elementary and left to the reader.
\end{proof}

Figure \ref{sqd_figure} shows a standard quadratic domain with its boundary $\phi_C(i\RR)$.
From now on, I denote by $\phi_C$ the restriction of $\phi_C$ to the closed right half-plane $\bar{H(0)}$.

Two domains $\Omega, \Delta \subseteq H(0)$ are \textbf{equivalent} if there exists $R>0$ such that $\Omega \cap D(R) = \Delta \cap D(R)$, where $D(R):= \set{z:\ |z| > R}$.  The corresponding equivalence classes of domains in $H(0)$ are called \textbf{germs at $\infty$} of domains in $H(0)$.  If clear from context, we shall not explicitely distinguish between a domain in $H(0)$ and its germ at $\infty$.

For $A \subseteq \CC$ and $\epsilon > 0$, let $$T(A,\epsilon):= \set{z \in \CC:\ d(z,A) < \epsilon}$$ be the \textbf{$\epsilon$-neighbourhood} of $A$.

\subsection*{Convention.}  Given a standard quadratic domain $\Omega$ and a function $g:\RR \into \RR$ that has a holomorphic extension on $\Omega$, I will usually denote this extension by the corresponding boldface letter $\gg$.  I also write $\bexp$ and $\xx$ for the holomorphic extensions on $\Omega$ of $\exp$ and the identity function $x$, respectively, and $\blog$ for the principal branch of $\log$ on $\Omega$.  Thus, every $m \in L$ has a unique holomorphic extension $\mm$ on $\Omega$.  (Strictly speaking, these extensions depend on $\Omega$, but I do not indicate this dependence.)

\begin{lemma}
	\label{sqd_facts}
	Let $C>0$.  The following inclusions hold as germs at $\infty$ in $H(0)$:
	\begin{enumerate}
		\item for $D > C$ and $\epsilon > 0$, we have $T(\Omega_D,\epsilon) \subseteq \Omega_C$;
		\item for $\nu>0$, we have $$\nu \cdot \Omega_C \subseteq \begin{cases} \Omega_{\nu C} &\text{if } \nu \le 1, \\ \Omega_C &\text{if } \nu \ge 1; \end{cases}$$ 
		\item for any standard quadratic domain $\Omega$,  we have $\blog(\Omega_C) \subseteq \Omega$;
		\item we have $\Omega_C + \Omega_C \subseteq \Omega_C$.  
	\end{enumerate}
\end{lemma}

\begin{proof}
	(1) follows from Lemma \ref{sqd_lemma}(3).
	
	(2) follows from Lemma \ref{sqd_lemma}(3) and the equality $$\nu\cdot(x, 2(x/C)^2) = (\nu x, 2(\nu x/\sqrt{\nu}C)^2)$$ in $\RR^2$.
	
	(3) Note that $\blog(H(0) \cap \{|z| > 1\}) = H(0) \cap \set{|\im z| < \pi/2}$.
	
	(4) Note first that, for $a \in \CC$ with $\re a \ge 0$, the boundary of $a + \Omega_C$ in $\set{z \in \CC:\ \im z \ge \im a}$, viewed as a subset of $\RR^2$, is the graph of a function $f_{a,C}:[C + \re a,+\infty) \into [\im a,+\infty)$ such that $$f_{a,C}(x) \sim \im a + \left(\frac{x- \re a}{C}\right)^2.$$  In particular, if $a \in \partial \Omega_C$, then $a = b + if_C(b)$ for some $b \ge C$; therefore, $$f_{a,C}(x) \sim \frac {b^2 + (x-b)^2}{C^2} < f_C(x)$$ in $\C$, which proves the claim.
\end{proof}

The following is the main reason for working with standard quadratic domains.

\begin{lemma}
	\label{exp_on_sqd}
	Let $C>0$ and set $K:= \frac{C}{\sqrt 3}$.  There exists $k \in (0,1)$ depending on $C$ such that $$k\exp\left(K\sqrt{|z|}\right) \le \left|\bexp(z)\right| \le \exp(|z|)$$ for $z \in \Omega_C$.
\end{lemma}

\begin{proof}
	Let $C>0$ be such that $\Omega = \Omega_C$ and, for $r>0$, denote by $C_r$ the circle with center 0 and radius $r$.  Since $|\bexp(x+iy)| = \exp x$, the point in $\Omega \cap C_r$ where $|\bexp z|$ is maximal is $z = r$.  On the other hand, the point $z(r)  = x(r)+iy(r)$ in $\Omega \cap C_r$ where $|\bexp z|$ is smallest lies on the boundary of $\Omega_r$, so that $y(r) = f_C(x(r))$.  It follows from Lemma \ref{sqd_lemma}(3) that $$r = \sqrt{x(r)^2 + f_C(x(r))^2} \sim x(r)^2 \sqrt{\frac1{x(r)^2} + \frac{4}{C^4}}.$$  Hence $x(r) \ge K\sqrt r$ for all sufficiently large $r \in \RR$, as required.
\end{proof}

\subsection*{Convention} Given an unbounded domain $\Omega \subseteq H(0)$ and holomorphic $\phi,\psi:\Omega \into \CC$, I write $$\psi = o(\phi) \quad\text{in } \Omega$$ if $|\psi(z)/\phi(z)| \to 0$ as $|z| \to \infty$ in $\Omega$.  
\medskip

The reason why the notion of qaa algebra makes sense for the set of monomials $L$ is that, for $m,n \in L$, we have $m < n$ if and only if $m = o(n)$.  This equivalence remains true on standard quadratic domains:
 
\begin{lemma}
	\label{strong_order_cor}
	Let $m,n \in L$ be such that $m < n$, and let $\Omega$ be a standard quadratic domain.  Then
	$\mm = o(\nn)$ in $\Omega$. 
\end{lemma}

\begin{rmk}
	While $\exp^{-1} < x^{-1}$ in $L$, we have $\bexp^{-1} \neq o(\xx^{-1})$ in $H(0)$ (or indeed in any right half-plane).
\end{rmk}

\begin{proof}
	First, let $z \in H(0)$ with $|z| \ge e$.  Then $$1 \le \log|z| = \re(\blog z) \le |\blog z|$$ and, since $\im(\blog z) \in (-\frac\pi2,\frac\pi2)$, we also have $$|\blog z| \le 3\log|z|.$$
	
	Second, define $e_0:= 1$ and, for $k > 0$, we set $e_k:= e^{e_{k-1}}$.  It follows from (1), by induction on $k \in \NN$, that if $z \in H(0)$ with $|z| \ge e_k$, there exists $C = C(k) > 0$ such that $$0 \le \log_k|z| \le |\blog_k z| \le C \log_k|z|.$$
	
	The previous two observations, together with Lemma \ref{exp_on_sqd} and the characterization of the ordering of $L$ given in the introduction, imply that if $m \in L$ is such that $m < 1$, then $\mm = o(\mathbf 1)$ in $\Omega$.  Since $L$ is a multiplicative group, the lemma follows.
\end{proof}

\section{Strong asymptotic expansions}\label{asymptotic_section}

Set $E:= \set{\exp^r:\ r \in \RR}.$  Note that $E$ is co-initial in $L$; in particular, a series $F \in \Gs{\RR}{E}$ has $E$-natural support if and only if it has $L$-natural support.

\begin{df}
	\label{asymptotic_def}
	Let $f\in \C$ and $F = \sum f_r \exp^{-r} \in \Gs{\C}{E}$.  The germ $f$ has \textbf{strong asymptotic expansion $F$ (at $\infty$)} if 
	\begin{renumerate}
		\item $F$ has $E$-natural support;
		\item $f$ has a holomorphic extension $\ff$ on some standard quadratic domain $\Omega$;
		\item each $f_r$ has a holomorphic extension $\ff_r$ on $\Omega$ such that $\ff_r = o(\bexp^{s})$ in $\Omega$, for each $s>0$;
		\item for each $r \in \RR$, we have 
		\begin{equation}
		\label{strong_asymptotics}\tag{$\ast_{f,r}$}
		\ff - \sum_{s \le r} \ff_s \bexp^{-s} = o\left(\bexp^{-r}\right) \quad\text{in } \Omega.
		\end{equation}
	\end{renumerate}
	In this situation, $\Omega$ is called a \textbf{strong asymptotic expansion domain} of $f$.  
\end{df}	

\begin{expl}
	\label{A0-rmk}
	Let $f \in \C$ be almost regular with asymptotic expansion $F:= \sum_{n=0}^\infty p_n \exp^{-\nu_n}$ as defined in the introduction.  Then $F$ is a strong asymptotic expansion of $f$.
	
	To see this, let $r \in \RR$; Condition \eqref{strong_asymptotics} holds by definition if $r = \nu_N$ for some $N \in \NN$, so  assume that $\nu_{N-1} < r < \nu_N$ for some $N$ (setting $\nu_{-1}:= -\infty$ to make sense of all cases).  The definition of ``almost regular'' implies that
	\begin{equation*}
	\ff - \sum_{\nu_n \le r} \pp_n \bexp^{-\nu_n} - \pp_N \bexp^{-\nu_N} = o\left(\bexp^{-\nu_N}\right) \quad\text{in } \Omega.
	\end{equation*}
	Condition \eqref{strong_asymptotics} now follows, because $|z| \to \infty$ in $\Omega$ implies $\re z \to +\infty$, so that $\qq\bexp^{-\nu_N} = o\left(\bexp^{-r}\right)$ in $\Omega$, for every polynomial $q$.
\end{expl}

\begin{rmk}
	\label{asym_expansion_rmk}
	Let $f \in \C$ have strong asymptotic expansion $F \in \Gs{\C}{E}$, and let $s \in \RR$.  Then $f\cdot \exp^s$ has strong asymptotic expansion $F\cdot \exp^s$.
\end{rmk}

\begin{lemma}
	\label{closure_lemma}
	Let $f,g \in \C$ have strong asymptotic expansions $\sum a_s \exp^{-s}$ and $\sum b_s \exp^{-s}$, respectively, in a standard quadratic domain $\Omega$.  Then
	\begin{enumerate}
		\item $f+g$ has strong asymptotic expansion $\sum (a_s+b_s) \exp^{-s}$ in $\Omega$;
		\item $fg$ has strong asymptotic expansion $\left(\sum a_s \exp^{-s}\right) \left(\sum b_s \exp^{-s}\right)$ in $\Omega$;
		\item if $f=0$ and $s_0:= \min\{s \in \RR:\ a_s \ne 0\}$, then there exists $r>0$ such that  $\aa_{s_0} = o\left(\bexp^{-r}\right)$ in $\Omega$.
	\end{enumerate}
\end{lemma}

\begin{proof}
	Fix $r \ge 0$. Then in $\Omega$,
	\begin{align*}
		\ff+\gg &- \sum_{s \le r} (\aa_s+\bb_s)\bexp^{-s} \\&= \left( \ff - \sum_{s \le r} \aa_s \bexp^{-s} \right) + \left( \gg - \sum_{s \le r} \bb_s \bexp^{-s} \right) \\&= o\left(\bexp^{-r}\right),
	\end{align*}
	which proves (1).  For (2), write $$\sum c_s \exp^{-s} = \left(\sum a_s \exp^{-s}\right) \left(\sum b_s \exp^{-s}\right),$$ so that $c_s = \sum_{s_1 + s_2 = s} a_{s_1}b_{s_2}.$  By the remark before this lemma, after replacing $f$ and $g$ by $f\exp^s$ and $g\exp^s$ for some $s \le 0$, I may assume that $a_s = b_s = 0$ for $s \le 0$; then $\ff$ and $\gg$, as well as $\aa_s\bexp^{-s}$ and $\bb_s\bexp^{-s}$ for each $s$, are bounded in $\Omega$.  Since
	\begin{align*}
		\ff\gg- &\sum_{s \le r} \cc_s \bexp^{-s} = \left(\ff - \sum_{s \le r} \aa_s \bexp^{-s}\right)\gg + \\&\quad+ \left(\sum_{s \le r} \aa_s \bexp^{-s}\right)\left(\gg - \sum_{s \le r} \bb_s \bexp^{-s}\right)+ \\&\quad+\left(\sum_{s \le r} \aa_s \bexp^{-s}\right)\left(\sum_{s \le r} \bb_s \bexp^{-s}\right) - \sum_{s \le r} \cc_s \bexp^{-s},
	\end{align*}
	it follows that the first and second of these four summands are $o\left(\bexp^{-r}\right)$ in $\Omega$.  As to the third and fourth summands, 
	\begin{multline*}
		\left(\sum_{s \le r} \aa_s \exp^{-s}\right)\left(\sum_{s \le r} \bb_s \bexp^{-s}\right) - \sum_{s \le r} \cc_s \bexp^{-s} \\= \left(\sum_{s \le r} \aa_s \bexp^{-s}\right)\left(\sum_{s \le r} \bb_s \bexp^{-s}\right) - \sum_{s_1+s_2 \le r} \aa_{s_1}\bb_{s_2} \bexp^{-s_1-s_2} \\ = \sum_{s_1, s_2 \le r \atop s_1+s_2>r} \aa_{s_1}\bb_{s_2} \bexp^{-s_1-s_2},
	\end{multline*}
	which is $o\left(\bexp^{-rx}\right)$ in $\Omega$, because the latter sum is finite.
	
	For (3) set $s_1:= \min\{s>s_0:\ a_s \ne 0\} > s_0$.  Then Condition \eqref{strong_asymptotics}, with $r:= \frac12(s_0+s_1)$, implies that $\aa_{s_0} \bexp^{-s_0} = o(\bexp^{-r})$ in $\Omega$, so that $\aa_{s_0} = o\left(\bexp^{-(r-s_0)}\right)$.
\end{proof}

For $F = \sum_{r \in \RR} f_r \exp^{-r} \in \Gs{\C}{E}$, I set $$\ord(F):= \min\set{r \in \RR:\ f_r \ne 0}.$$
Recall that, given series $F_n \in \Gs{\C}{E}$ for $n \in \NN$ such that $\ord(F_n) \to +\infty$ as $n \to \infty$, the infinite sum $\sum_n F_n$ defines a series in $\Gs{\C}{E}$.  The next criterion is useful for obtaining strong asymptotic expansions.

\begin{lemma}
	\label{inf_sum_of_strong}
	Let $f \in \C$ and $f_n \in \C$, for $n \in \NN$, and let $\Omega$ be a standard quadratic domain.  Assume that each $f_n$ has strong asymptotic expansion $F_n \in \Gs{\C}{E}$ in $\Omega$ such that $\ord(F_n) \to +\infty$ for $n \in \NN$, and assume that $f$ has a holomorphic extension $\ff$ on $\Omega$ such that $$\ff - \sum_{i=0}^n \ff_i = o(\ff_n) \quad\text{in } \Omega, \text{ for each } n.$$  Then the series $\sum_n F_n$ is a strong asymptotic expansion of $f$ in $\Omega$.
\end{lemma}

\begin{proof}
	Let $r \in \RR$, and choose $N \in \NN$ such that $\ord(F_n) >r$ for all $n \ge N$.  Then $\ff_n = o(\bexp^{-r})$ in $\Omega$, for $n \ge N$, so $$\ff - \sum_{i=0}^n \ff_i = o(\bexp^{-r}) \quad\text{in }\Omega.$$  Increasing $N$ if necessary, we may assume that $$\left(\sum_{i=0}^\infty F_i\right)_{\exp^{-r}} = \sum_{i=0}^N (F_i)_{\exp^{-r}}.$$  Therefore, with $\hh_r$ the holomorphic extension of $\left(\sum F_i\right)_{\exp^{-r}}$ on $\Omega$ and $\hh_{i,r}$ the holomorphic extension of $(F_i)_{\exp^{-r}}$ on $\Omega$, we get
	\begin{align*}
	\ff - \hh_r &= \ff - \sum_{i=0}^N \hh_{i,r} \\
	&= \left( \ff - \sum_{i=0}^N \ff_i \right) + \sum_{i=0}^N \left(\ff_i - \hh_{i,r}\right) \\ &= o(\bexp^{-r}) \quad\text{in } \Omega,
	\end{align*}
	as required.
\end{proof}

To extend the notion of strong asymptotic expansion to series in $\Gs{\RR}{L}$, I proceed as in Definition \ref{relative_asymptotic}:

\begin{df}
	\label{strong_qaa_algebra}
	Let $\K \subseteq \C$ be an $\RR$-algebra, let $L'$ be a divisible subgroup of $L$, and let $T:\K \into \Gs{\RR}{L'}$ be an $\RR$-algebra homomorphism.  We say that the triple $(\K,L',T)$ is a \textbf{strong qaa algebra} if 
	\begin{renumerate}
		\item $T$ is injective;
		\item the image $T(\K)$ is truncation closed;
		\item for every $f \in \K$, there exists a standard quadratic domain $\Omega$ such that $f$ and each $g_n:= T^{-1}((Tf)_n)$, for $n \in L'$, have holomorphic extensions $\ff$ and $\gg_n$ on $\Omega$, respectively, that satisfy
		\begin{equation}\label{strong_asymptotic_relation}
		\ff - \gg_n = o(\nn) \quad\text{in } \Omega.
		\end{equation}
	\end{renumerate}
	In this situation, I call $T(f)$ the \textbf{strong $\K$-asymptotic expansion} of $f$ and $\Omega$ a \textbf{strong $\K$-asymptotic expansion domain} of $f$.
\end{df}

\begin{lemma}
	\label{strong_full_asymptotics}
	Let $(\K,L',T)$ be a strong qaa algebra, with $L'$ a divisible subgroup of $L$.  Then $(\K,L,T)$ is a strong qaa algebra. 
\end{lemma}

\begin{proof}
	Let $f \in \K$ and $n \in L$; if $n \in L'$, then the asymptotic relation \eqref{strong_asymptotic_relation} holds by assumption, so assume $n \notin L'$.  If $n \le \supp(Tf)$, then $T^{-1}((Tf)_n) = f$, so the asymptotic relation \eqref{strong_asymptotic_relation} holds trivially.  So assume also that $n \not\le \supp(Tf)$ and choose the maximal $p \in \supp(Tf)$ such that $p<n$ (which exists because $\supp(Tf)$ is reverse well-ordered).  By assumption, writing $\gg_p$ and $\gg_n$ for the holomorphic extensions of $T^{-1}((Tf)_p)$ and $T^{-1}((Tf)_n)$, respectively,
	\begin{equation*}
	o(\pp) = \ff - \gg_p = \ff - \gg_n - a\pp,
	\end{equation*}
	for some nonzero $a \in \RR$.  Since $\pp = o(\nn)$ in $\Omega$ by Lemma \ref{strong_order_cor}, the asymptotic relation \eqref{strong_asymptotic_relation} follows.
\end{proof}

\section{The construction}\label{construction_section}

\subsection*{The initial Ilyashenko algebra}

In view of Fact \ref{ph-l} and in the spirit of \cite[Section 24]{Ilyashenko:2008fk}, I  define $\A_0$ to be the set of all $f \in \C$ that have a strong asymptotic expansion $F = \sum_{r \ge 0} a_r \exp^{-r} \in \Gs{\RR}{E}$.  Note that the condition $\supp(F) \subseteq [0,+\infty)$ implies that $f$ has a \textit{bounded} holomorphic extension to some standard quadratic domain.

\begin{lemma}
	\label{algebra_cor}
	\begin{enumerate}
		\item $\A_0$ is an $\RR$-algebra.
		\item Each $f \in \A_0$ has a unique strong asymptotic expansion $T_0f \in \Gs{\RR}{E}$.
		\item The map $T_0:\A_0 \into \Gs{\RR}{E}$ is an injective $\RR$-algebra homomorphism.  
	\end{enumerate}
\end{lemma}

\begin{proof}
	Part (1) follows from Lemma \ref{closure_lemma}(1,2).  For part (2), assume for a contradiction that $0$ has a nonzero strong asymptotic expansion $\sum a_r \exp^{-r} \in \Gs{\RR}{E}$ of order $s_0$.  Then by Lemma \ref{closure_lemma}(3), we have $a_{s_0} = o\left(\exp^{-r}\right)$ for some $r>0$; since $a_{s_0} \in \RR$, it follows that $a_{s_0} = 0$, a contradiction.  For part (3), the map $T_0$ is a homomorphism by Lemma \ref{closure_lemma}(1,2), and its kernel is trivial by Fact \ref{ph-l}.
\end{proof}

\begin{cor}
	\label{closure_cor}
	The triple $\left(\A_0,L,T_0\right)$ is a strong qaa algebra.  	
\end{cor}

\begin{proof}
By Lemma \ref{strong_full_asymptotics}, it suffices to show that $\left(A_0,E,T_0\right)$ is a strong qaa algebra.  For $r \ge 0$ the function $\exp^{-r}$ has a bounded holomorphic extension on $H(0)$, so it belongs to $\A_0$ with $T_0 \exp^{-r} = \exp^{-r}$.  Since the support of $T_0 f$, for $f \in \A_0$, is $E$-natural, every truncation of $T_0 f$ is an $\RR$-linear combination of $\exp^{-r}$, for various $r\ge0$, and therefore belongs to $\A_0$ as well.
\end{proof}
	
\begin{expls}\label{A00-expls}
	\begin{enumerate}
		\item Let $p \in \Ps{R}{X^*}$ be convergent with natural support \cite{Dries:1998xr,Kaiser:2009ud}.  Then $p \circ \exp^{-1} \in \A_0$.
		\item The algebra $\A_0 \circ (-\log)$ is the class $\A_1 = \A_1^{1,0}$ considered in \cite[Definition 5.4]{Kaiser:2009ud}.  In particular, for $f \in \A_0$ the series $$T_0(f) \circ (-\log) \in \Ps{R}{X^*}$$ has natural support and, for $r \ge 0$ and $g_r:= T_0^{-1}\left(T_0(f)\right)_{\exp^{-r}}$, we have
		$f(-\log x) - g_r(-\log x) = o(x^r)$ as $x \to 0^+$.
	\end{enumerate}
\end{expls}

\subsection*{The initial Ilyashenko field}

For  $f \in \A_0$, I set $$\ord(f):= \ord\left(T_0(f)\right).$$
Below, I call $f \in \C$ \textbf{infinitely increasing} if $f(x) \to +\infty$, \textbf{small} if $f(x) \to 0$ and a \textbf{unit} if $f(x) \to 1$, as $x \to +\infty$.

Similarly, let $G \in \Gs{\RR}{L}$, and let $g \in L$ be the leading monomial of $G$; so there are nonzero $a \in \RR$ and $\epsilon \in \Gs{\RR}{L}$ such that $G = ag(1+\epsilon)$. Note that the leading monomial of $\epsilon$ is small.  I call $G$ \textbf{small} if $g$ is small, and I call $G$ \textbf{infinitely increasing} if both $g$ is infinitely increasing and $a>0$.

\begin{rmkdf}
	\label{series_comp_1}
	Let $G \in \Gs{\RR}{L}$, and let $g \in L$ be the leading monomial of $G$; so there are nonzero $a \in \RR$ and small $\epsilon \in \Gs{\RR}{L}$ such that $G = ag(1+\epsilon)$.  Let also $k \in \{-1\} \cup \NN$ and $F \in \Gs{\RR}{(X_{-1}, \dots, X_k)^*}$ be such that $F$ has natural support and $$G = F\left(\frac1\exp, \frac1{\log_0}, \dots, \frac1{\log_k}\right).$$  Let $\alpha = (\alpha_{-1}, \dots, \alpha_k) \in \RR^{2+k}$ be the minimum of the support of $F$ with respect to the lexicographic ordering on $\RR^{2+k}$, so that $$g = \exp^{-\alpha_{-1}}\log_0^{-\alpha_0} \cdots \log_k^{-\alpha_k}.$$
	\begin{itemize}
		\item[Case 1:] Let $P \in \Ps{R}{X^*}$ be of natural support, and assume that $G$ is small.  Then $\alpha > (0, \dots, 0)$ in the lexicographic ordering of $\RR^{2+k}$.  
		\item[Case 2:] Let $P \in \Ps{R}{\left(\frac1X\right)^*}$ be of natural support, and assume that $G$ is infinitely increasing.  Then $\alpha < (0, \dots, 0)$ in the lexicographic ordering of $\RR^{2+k}$.  
	\end{itemize} 
	In both cases, $P \circ F$ belongs to $\Gs{\RR}{(X_{-1}, \dots, X_k)^*}$ and has natural support as well.  I therefore define $$P \circ G := (P \circ F)\left(\frac1\exp, \frac1{\log_0}, \dots, \frac1{\log_k}\right).$$
	This composition is associative in the following sense: whenever $P \in \Ps{R}{X^*}$ is small and of natural support and $Q \in \Ps{R}{X^*}$ is of natural support, then $Q \circ (P \circ G) = (Q \circ P) \circ G$.  A similar statement holds in Case 2; as usual, I will therefore simply write $Q \circ P \circ G$ for these compositions.
\end{rmkdf}

\begin{lemma}
	\label{construction_lemma_1}
	Let $f,g \in \A_0$, and set $d:= \ord(g) \ge 0$.
	\begin{enumerate}
		\item There exist unique nonzero $g_{d} \in \RR$ and $\epsilon \in \A_0$ such that $$g = g_{d} \exp^{-d} (1 - \epsilon)$$ and $\ord(\epsilon) > 0$.
		\item Assume that $g$ is small with strong asymptotic expansion domain $\Omega$, and let $P \in \Ps{R}{X}$ be convergent.  Then $P \circ g$ belongs to $\A_0$, has strong asymptotic expansion domain $\Omega$ and satisfies $T_0(P \circ g) = P \circ T_0(g)$.
	\end{enumerate}
\end{lemma}

\begin{rmk}
	In the situation of Part (1) above, the germ $\frac g{g_d \exp^{-d}}$ is a unit belonging to $\A_0$.
\end{rmk}

\begin{proof}
	(1) Say $T_0(g) = \sum_{r \ge d} g_r \exp^{-r}$; then  take $$\epsilon := -\frac{g - g_d\exp^{-d}}{g_d\exp^{-d}},$$ which belongs to $\A_0$ by Lemma \ref{closure_lemma}(2).
	
	(2) By Condition $(\ast_{g,0})$, the function $P \circ \gg$ is a bounded, holomorphic extension of $P \circ g$ on $\Omega$.  Moreover, say $P(X) = \sum a_nX^n \in \Ps{R}{X}$; since $P(z) - \sum_{i=0}^n a_i z^i = O(z^n)$ at 0 in $\CC$ by absolute convergence, it follows that $$P \circ \gg - \sum_{i=0}^n \gg^i = o(\gg^n) \quad\text{in } \Omega.$$  From Lemma \ref{closure_lemma}, it follows that $a_n g^n \in \A_0$ has strong asymptotic expansion domain $\Omega$ and satisfies $$T_0(a_n g^n) = a_n T_0(g)^n,$$ for each $n$.  Since $g$ is small, we have $d>0$, so we also get $\ord\left(g^n\right) = ns \to \infty$ as $n \to \infty$.  Part (2) now follows from Lemma \ref{inf_sum_of_strong}.
\end{proof}

Let $\F_0$ be the fraction field of $\A_0$ and extend $T_0$ to an $\RR$-algebra homomorphism $T_0:\F_0 \into \Gs{\RR}{E}$ in the obvious way (also denoted by $T_0$).  Note that the  functions in $\F_0$ do not all have \textit{bounded} holomorphic extensions to standard quadratic domains; hence the need for first defining $\A_0$.

\begin{rmk}
	Let $\K$ be a subfield of $\C$.  Let $F,G \in \Gs{\K}{E}$, let $g$ be the leading term of $G$ and set $\epsilon:= -\frac{G-g}{g}$.  Recall that $$\frac FG = \frac Fg \cdot (\geom\circ\epsilon),$$ where $\geom = \sum_{n=0}^\infty X^n$ is the geometric series.
\end{rmk}

\begin{cor}
	\label{construction_cor_1}
	\begin{enumerate}
		\item Let $f \in \F_0$.  Then $f$ has strong asymptotic expansion $T_0\left(f\right)$, and there exist unique $d, f_d \in \RR$ and $\epsilon \in \A_0$ such that $$f = f_d \exp^{-d}(1+\epsilon)$$ and  $\ord(\epsilon) > 0$.  
		\item $\left(\F_0,L,T_0\right)$ is a strong qaa field.
	\end{enumerate}
\end{cor}

\begin{proof}
	(1) Say $f = g/h$, for some $g,h \in \A_0$ with $h \ne 0$ of order $s \ge 0$.  By Lemma \ref{construction_lemma_1}(1) there are $h_s \in \RR\setminus\{0\}$ and $\epsilon \in \A_0$ such that $h = h_s \exp^{-s}(1-\epsilon)$ and $\ord(\epsilon) > 0$.  In particular, $\epsilon$ is small, so that $$f = \frac{g}{h_s \exp^{-s}(1-\epsilon)} = \frac {\exp^s}{h_s} g \geom(\epsilon).$$  Part (1) now follows from Lemmas \ref{closure_lemma} and \ref{construction_lemma_1}(2).
	
	Since the series in $T_0\left(\F_0\right)$ have $E$-natural support and each monomial in $E$ belongs to $\F_0$, the triple $\left(\F_0,E,T_0\right)$ is a qaa field.  Part (2) now follows from Lemma \ref{strong_full_asymptotics}.
\end{proof}

\subsection*{Iteration}

I construct strong qaa fields $\left(\F_k,L,T_k\right)$, for nonzero $k \in \NN$, such that $\F_{k-1}$ is a subfield of $\F_{k}$ and $T_{k}$ extends $T_{k-1}$,  which I summarize by saying that $\left(\F_k,L,T_k\right)$ \textbf{extends} $\left(\F_{k-1},L,T_{k-1}\right)$.  As in the initial stage of the construction, I will obtain $\F_k$ as the fraction field of a strong qaa algebra $\left(\A_k, L, T_k\right)$ such that 
\begin{renumerate}
	\item each $f \in \A_k$ has a bounded, holomorphic extension to some standard quadratic domain;
	\item for each $f \in \F_k$, there exists $s \in \RR$ such that $\frac f{\exp^s}$ belongs to $\A_k$.
\end{renumerate}
Note that, by Lemma \ref{construction_lemma_1}(1), conditions (i) and (ii) hold for $k=0$, provided I set $\A_{-1} = \F_{-1}:= \RR$.

The construction proceeds by induction on $k$; the case $k=0$ is done above.  So assume $k>0$ and that $\left(\A_{i},L,T_{i}\right)$ and $\left(\F_{i},L,T_{i}\right)$ have been constructed, for $i=0, \dots, k-1$.  First, I set
\begin{equation*}
\F'_{k} := \F_{k-1} \circ \log
\end{equation*}
and define $T'_{k}:\F'_{k} \into \Gs{\RR}{L'}$ by $$T'_{k}(f \circ \log) := \left(T_{k-1} f\right) \circ \log,$$ where $$L':= \set{m \in L:\ m(-1) = 0}.$$

\begin{cor}
	\label{field_closure_cor}
	$\left(\F'_{k},L,T'_{k}\right)$ is a strong qaa field.
\end{cor}

\begin{proof}
	Since $\blog$ maps $H(0)$ into any standard quadratic domain, the triple $\left(\F'_{k},L',T'_{k}\right)$ is a strong qaa field.  Since $L'$ is a divisible subgroup of $L$, the corollary follows from Lemma \ref{strong_full_asymptotics}.
\end{proof}

\begin{nrmk}
	\label{rel_growth_rmk}
	Let $g \in \F'_{k}$.  There exists, by condition (ii) above, an $s \in \RR$ such that $g/x^s$ has a bounded holomorphic extension on some standard quadratic domain $\Omega$.  Thus $g = o(\exp^r)$ for every $r>0$ and, since $\F'_k$ is a field, it follows that $g = o(\exp^{-r})$ for some $r>0$ if and only if $g = 0$.
\end{nrmk}

Now let $\A_{k}$ be the set of all $f \in \C$ that have a bounded, holomorphic extension on some standard quadratic domain $\Omega$ and a strong asymptotic expansion $\sum_{r \ge 0} f_r \exp^{-r} \in \Gs{\F'_k}{E}$ in $\Omega$.  (The boundedness assumption is included here, because not all $f \in \F'_k$ are bounded if $k \ge 0$.)

By Remark \ref{rel_growth_rmk}, arguing as in Lemma \ref{algebra_cor}, we see that  $\A_{k}$ is an $\RR$-algebra, each $f \in \A_{k}$ has a unique strong asymptotic expansion $$\tau_k f := \sum_{r \ge 0} f_r \exp^{-r} \in \Gs{\F'_{k}}{E},$$ and the map $\tau_k:\A_{k} \into \Gs{\F'_{k}}{E}$ is an $\RR$-algebra homomorphism.  Moreover, it follows from Fact \ref{ph-l} that this map is injective.  
For $f \in \A_{k}$ with $\tau_k f = \sum f_r \exp^{-r}$, I now define $$T_{k} f:= \sum_{r \ge 0} (T'_{k} f_r) \exp^{-r}.$$ 
For completeness' sake, I also set $\tau_0:= T_0$.

\begin{prop}
	\label{closure_cor_2}
	The triple $\left(\A_{k},L,T_{k}\right)$ is a strong qaa algebra that  extends $\left(\A_{k-1},L,T_{k-1}\right)$.  
\end{prop}

\begin{proof}
	The map $\sigma:\Gs{\F'_{k}}{E} \into \Gs{\RR}{L}$ defined by $$\sigma\left( \sum f_r \exp^{-r} \right) := \sum (T'_{k} f_r) \exp^{-r}$$ is an $\RR$-algebra homomorphism, and it is injective because $T'_{k}$ is injective.  Since $T_{k} = \sigma \circ \tau_k$, it follows that $T_{k}$ is an injective $\RR$-algebra homomorphism.
	Let now $f \in \A_k$ be such that $$T_k f = \sum_{m \in L} a_m m \quad\text{and}\quad \tau_k f = \sum_{r \ge 0} f_r \exp^{-r},$$ and let $n \in L$; we show there exists $g \in \A_k$ such that $T_k g = (T_k f)_n$.  Considering $n$ as a function $n:\{-1\} \cup \NN \into \RR$, set $r:= -n(-1)$ and $n':= \prod_{i=0}^\infty \log_i^{n(i)} \in L'$, so that $n = n'\exp^{-r}$ and
	\begin{equation*}
	\left(T_kf\right)_n = \sum_{m(-1) > n(-1)} a_m m + \left(T'_k f_r\right)_{n'} \exp^{-r},
	\end{equation*}
	and let $\Omega$ be a strong asymptotic expansion domain of $f$.
	Note that each $f_s \exp^{-s}$ has a bounded holomorphic extension on $\Omega$.  Since $$\sigma^{-1}\left(\sum_{m(-1) > n(-1)} a_m m\right) = \sum_{s < r} f_s \exp^{-s}$$ has finite support in $\Gs{\F'_k}{E}$, it follows that $$g_1:= \sum_{s<r} f_s\exp^{-s}$$ belongs to $\A_k$ and satisfies $\tau_k g_1 = g_1$ and $T_k g_1 =  \sum_{m(-1) > n(-1)} a_m m$.  On the other hand, by the inductive hypothesis, there exists $h \in \F'_k$ such that $T'_k h = \left(T'_k f_r\right)_{n'}$.  Hence $h \exp^{-r} \in \A_k$ and, by definition of $T_k$, we obtain $T_k(h \exp^{-r}) = \left(T'_k f_r\right)_{n'} \exp^{-r}$.  Therefore, we can take $g:= g_1+h \exp ^{-r}$.
	
	Finally, after shrinking $\Omega$ if necessary, we may assume that $\Omega$ is also a strong asymptotic expansion domain of $g$; we now claim that $\ff-\gg = o(\nn)$ in $\Omega$, which then proves the proposition.  By the inductive hypothesis, we have $\ff_r - \hh = o(\nn')$ in $\Omega$; therefore, 
	\begin{equation} \label{asym_1}
	\ff_r \bexp^{-r} - \hh \bexp^{-r} = o(\nn) \quad\text{ in } \Omega.
	\end{equation} 
	On the other hand, let $r':= \min\set{s \in \RR:\ s > r \text{ and } f_r \ne 0}$.  Then, by hypothesis, we have 
	\begin{equation} \label{asym_2}
	\ff - \gg_1 - \ff_r \bexp^{-r} = o\left(\bexp^{-\frac{r+r'}2}\right) \quad\text{ in } \Omega.
	\end{equation}
	Since $\bexp^{-\frac{r+r'}2} = o(\nn)$ in $\Omega$, the proposition follows.
\end{proof}

Next, identify $\Gs{\RR}{L}$ with a subset of $\Gs{\Gs{\RR}{L'}}{E}$ in the obvious way, and for $F \in \Gs{\Gs{\RR}{L'}}{E}$ set $$\ord(F):= \min\supp(F).$$
Note that $\ord(\tau_k(f)) = \ord(T_k(f))$ for $f \in \F_k$, so I set $$\ord(f):= \ord(\tau_k(f)).$$

Let $P \in \Ps{R}{X^*}$ have natural support, and let $G \in \Gs{\F_k'}{E}$ be such that $\ord(G) > 0$.  Then there exists $F \in \As{\F_k'}{X^*}$ such that $F$ has natural support, $\ord(F) > 0$ and $G = F\left(\exp^{-1}\right).$  Hence $P \circ F$ belongs to $\As{\F_k'}{X^*}$ and has natural support as well.  We therefore define $$P \circ G := (P \circ F)\left(\exp^{-1}\right),$$ which belongs to $\Gs{\F_k'}{E}$.  Similar to the situation in Remark and Definition \ref{series_comp_1}, this composition is associative: if $\ord(P)>0$ and $Q \in \Ps{R}{X^*}$ has natural support, then $(Q \circ P) \circ F = Q \circ (P \circ F)$.

\begin{lemma}
	\label{construction_lemma}
	Let $g \in \A_k$, and set $d:= \ord(g) \ge 0$.
	\begin{enumerate}
		\item There exist unique $g_{d} \in \F'_k$ and $\epsilon \in \A_k$ such that $$g = g_{d} \exp^{-d} (1 + \epsilon)$$ and $\ord(\epsilon) > 0$.
		\item Assume $\ord(g)>0$, and let $P \in \Ps{R}{X}$ be convergent.  Then $P \circ g \in \A_k$, and we have $\tau_k(P \circ g) = P \circ \tau_k(g)$ and $T_k(P \circ g) = P \circ T_k(g)$.
	\end{enumerate}
\end{lemma}

\begin{proof}
	Replacing $T_0$ by $\tau_k$ throughout, the proof of Lemma \ref{construction_lemma_1}(1,2) gives everything except the statement $T_k(P \circ g) = P \circ T_k(g)$.  However, in the the situation of part (2) with the notations from the proof of Lemma \ref{construction_lemma_1}(2), since for each $r \ge 0$ there exists $N_r \in \NN$ such that  $$(P \circ \tau_k(f))_{\exp^{-r}} = \sum_{n=0}^{N_r} a_n \left(\tau_k(f)^n\right)_{\exp^{-r}},$$ it follows that $\sigma(P \circ \tau_k(f)) = P \circ \sigma(\tau_k(f))$.  
\end{proof}

As in the construction of $\F_0$, I now let $\F_{k}$ be the fraction field of $\A_{k}$ and extend $\tau_k$ and $T_{k}$ correspondingly.

\begin{cor}
	\label{construction_cor}
	\begin{enumerate}
		\item Let $f \in \F_k$.  Then $f$ has strong asymptotic expansion $\tau_k\left(f\right)$, and there exist unique $d \in \RR$, $f_d \in \F'_k$ and $\epsilon \in \A_k$ such that $$f = f_d \exp^{-d}(1+\epsilon)$$ and $\ord(\epsilon) > 0$.  In particular, $f \in \A_k$ if and only if $f$ is bounded.
		\item The triple $\left(\F_k,L,T_k\right)$ is a strong qaa field.
	\end{enumerate}
\end{cor}

\begin{proof}
	(1) Say $f = g/h$, for some $g,h \in \A_k$ with $h \ne 0$ of order $s \ge 0$.  By Lemma \ref{construction_lemma}(1), there are nonzero $h_s \in \F'_k$ and $\epsilon \in \A_k$ such that $h = h_s \exp^{-s}(1-\epsilon)$ and $\ord(\epsilon) > 0$.  In particular, $\epsilon$ is small, so that $$f = \frac{g}{h_s \exp^{-s}(1-\epsilon)} = \frac {\exp^s}{h_s} g \geom(\epsilon).$$  Part (1) now follows from Lemmas \ref{closure_lemma} and \ref{construction_lemma}(2).
	
	(2) The map $T_k$ is injective, because the restriction of $T_k$ to $\A_k$ is.  Also, by part (1), each $f \in \F_k$ is of the form $f = \exp^r g$ with $g \in \A_k$ and $r \in \RR$.  Since $\left(\A_k,L,T_k\right)$ is a strong qaa algebra, it follows that $\left(\F_k,L,T_k\right)$ is a strong qaa field.
\end{proof}

\begin{nrmk}
\label{iterated_rmk}
Since $\A_0$ contains all polynomials in $\exp$, the algebra $\A_1$ contains the class $\A$ of almost regular maps.
\end{nrmk}

In view of Proposition \ref{closure_cor_2} and Corollary \ref{construction_cor}, we set $$\A:= \bigcup_k \A_k \quad\text{and}\quad \F:= \bigcup_k \F_k,$$ and we let $T$ be the common extension of all $T_k$ to $\F$; we denote the restriction of $T$ to $\A$ by $T$ as well.  It follows that $(\A,L,T)$ is a strong qaa algebra and $(\F,L,T)$ is a strong qaa field such that $\F$ is the fraction field of $\A$.  This finishes the proof of Theorem \ref{main_thm}(1).

\section{Closure under differentiation}  \label{diff_section}

The next lemma is a version of L'H\^opital's rule for holomorphic maps on standard quadratic domains.

\begin{lemma}
	\label{derivative_o}
	Let $0 < C < D$ and $\phi:\Omega_C \into \CC$ be holomorphic.
	\begin{enumerate}
		\item Let $r \in \RR$ be such that $\phi = o(\exp^{-r})$ in $\Omega_C$.  Then $\phi' = o(\exp^{-r})$ in $\Omega_D$.
		\item If $\phi$ is bounded in $\Omega_C$, then $\phi'$ is bounded in $\Omega_D$.
	\end{enumerate}
\end{lemma}

\begin{proof}
(1) By Lemma \ref{sqd_facts}(1), there is $R>0$ such that $D(z,2) \subseteq \Omega_C$ for every $z \in \Omega_D$ with $|z| > R$.  Let $z \in \Omega_D$ be such that $|z|>R$, and let $w_z \in \set{w:\ |w-z| = 1}$ be such that $|\phi(w_z)| = \max_{|w-z|=1} |\phi(w)|$; then, by Cauchy's formula, we have $$|\phi'(z)| \le |\phi(w_z)|.$$  On the other hand,  $$|e^{-rz}| = e^{-r\re z} \ge \begin{cases} e^{-r(\re w_z -2)} = e^{2r}e^{-rw_z} &\text{if } r \le 0, \\ e^{-r(\re w_z +2)} = e^{-2r}e^{-rw_z} &\text{if } r \ge 0. \end{cases}$$  Therefore,  $$\left|\frac{\phi'(z)}{e^{-rz}}\right| \le e^{2|r|}\left|\frac{\phi(w_z)}{e^{-rw_z}}\right|.$$  Since $|w_z| \sim |z|$ and $\phi = o(\exp^{-r})$ in $\Omega_C$, the conclusion follows.

The proof of (2) is similar and left to the reader.
\end{proof}

I now set $$\D:= \set{f \in \C:\ f \text{ is differentiable}}$$  and for $F = \sum f_r \exp^{-r} \in \Gs{\D}{E}$, I define $$F':= \sum (f'_r-rf_r) \exp^{-r} \in \Gs{\C}{E}.$$

\begin{prop}
	\label{derivative_prop}
	Let $k \in \NN$ and $f \in \F_k$.  Then $f' \in \F_k$ and  $\tau_k(f') = (\tau_kf)'.$
\end{prop}

\begin{proof}
	By induction on $k$; let $\tau_k(f) = \sum f_r \exp^{-r}$.  If $k=0$, then $(\tau_kf)' \in \Gs{\F'_k}{E}$ because the coefficients of $\tau_k f$ are real numbers.  If, on the other hand, $k>0$, then $f_r = g_r \circ \log$ for some $g_r \in \F_{k-1}$, so that 
	\begin{equation*}
	f_r' = \frac{g_r' \circ \log}{x} = \frac{g_r'}{\exp} \circ \log \in \F'_k
	\end{equation*}
	by the inductive hypothesis, so that again $(\tau_kf)' \in \Gs{\F'_k}{E}$.
	
	To finish the proof of the proposition, we may assume (by the quotient formula for derivatives) that $f \in \A_k$.
	Let $C>0$ be such that $\Omega_C$ is a domain of strong asymptotic expansion of $f$, and let $D>C$.  By Lemma \ref{derivative_o}(2), the map $\ff':\Omega_D \into \CC$ is a bounded, holomorphic extension of $f'$.  Moreover, if $r  \ge 0$, then
	\begin{equation*}
	\ff' - \sum_{s \le r} (\ff_s'-s\ff_s)\bexp^{-s} = \left(\ff - \sum_{s \le r} \ff_s\bexp^{-s}\right)' = o(\bexp^{-r}) \quad\text{in } \Omega_D,
	\end{equation*}
	by Lemma \ref{derivative_o}(1) and Condition $(\ast_{f,r})$, so that $f' \in \A_k$.
\end{proof}

Finally note that, for $m \in L$, the derivative $m'$ is a linear combination of elements of $L$ such that $\max\supp(m') \to 0$ as $m \to 0$ in $L$.  Therefore, for $F = \sum a_m m \in \Gs{\RR}{L}$, I define $$F':= \sum a_m m',$$ and I note that the map $F \mapsto F'$ is a derivation on $\Gs{\RR}{L}$.

\begin{cor}
	\label{derivative_cor}
	$\F$ is closed under differentiation and for $f \in \F$, we have $T(f') = (Tf)'$.
\end{cor}

\begin{proof}
	Let $k \in \NN$ and $f \in \F_k$; I proceed by induction on $k$ to show that $T(f') = (Tf)'$.  If $k=0$, then $T(f) = \tau_0(f)$ and $(Tf)' = (\tau_0f)'$, so the claim follows from Proposition \ref{derivative_prop} in this case.  So I assume $k>0$ and the claim holds for lower values of $k$.
	
	Say $\tau_k(f) = \sum f_r \exp^{-r}$; then $T(f) = \sum T(f_r) \exp^{-r}$ by definition, while $\tau_k(f') = (\tau_kf)' = \sum (f_r'-rf_r) \exp^{-r}$.  It follows from the inductive hypothesis that
	\begin{align*}
	T(f') &= \sum T(f_r'-rf_r) \exp^{-r} \\&= \sum \left( T(f_r') - r T(f_r) \right) \exp^{-r} \\&= \sum \left( (Tf_r)' - rT(f_r) \right) \exp^{-r} \\ &= (Tf)',
	\end{align*}
	as claimed.
\end{proof}

\section{Closure under $\log$-composition} \label{comp_section}

Note that, since $\F$ is a field, it is closed under $\log$-composition if and only if for all $f,g \in \F$ such that $\lim_{x \to +\infty} g(x) = +\infty$, the composition $f \circ \log \circ g$ belongs to $\F$.
First I show that, for infinitely increasing $g \in \F$, the map $\log \circ g$ always has a holomorphic extension that maps standard quadratic domains \textit{into} standard quadratic domains.

\begin{lemma}
	\label{large_asymptotics}
	Let $g \in \F$ and $\Omega_g$ be a strong $\F$-asymptotic expansion domain of $g$, and assume that $g$ is infinitely increasing.  Then, for some standard quadratic domain $\Omega'_g \subseteq \Omega_g$, the function $\log \circ g$ has a holomorphic extension $\ll_g$ on $\Omega'_g$ such that, for every standard quadratic domain $\Omega$, there exists a standard quadratic domain $\Delta \subseteq \Omega'_g$ with $(\ll_g)(\Delta) \subseteq \Omega$.
\end{lemma}

\begin{proof}
	Let $a>0$, $m \in L$ be the leading monomial of $\F$ and small $\epsilon \in \F$ be such that $g = am(1+\epsilon)$.  Shrinking $\Omega_g$ if necessary, I may assume that $\Omega_g$ is also a strong $\F$-asymptotic expansion domain of $\epsilon$ with corresponding holomorphic extension $\ee:\Omega_g \into \CC$.  Then by the asymptotic relation \eqref{strong_asymptotic_relation}, we have $$\gg = a\mm(1+\ee) \quad\text{with } \ee = o(\mathbf 1) \text{ in } \Omega_g;$$ in particular, after shrinking $\Omega_g$ again if necessary, the function $\log a + \log(1+\epsilon)$ has holomorphic extension $\blog \mathbf a + \blog(1+\ee)$ on $\Omega_g$ such that $\blog(1+\ee) = o(\mathbf 1)$ in $\Omega_g$.
	Since $$\log \circ g = \log a + \log \circ m + \log(1+\epsilon),$$ I may therefore assume by Lemma \ref{sqd_facts} that $g = m \in L$.  However $\log \circ m$ is an $\RR$-linear combination of $\log_i$, for various $i \in \NN$.  Let $i_0$ be the smallest $i$ such that $\log_i$ appears in this $\RR$-linear combination.  Since $m$ is infinitely increasing, the coefficient of $\log_{i_0}$ in this $\RR$-linear combination must be positive.  Since $\log_i = o(\log_{i_0})$ in $H(0)$, for $i > i_0$, it follows as above that I may even assume that $m = \log_{i_0}$.  But this last case follows from Lemma \ref{sqd_facts}(3).
\end{proof}

\subsection*{Formal $\log$-composition in $\Gs{\RR}{L}$}
Let $G \in \Gs{\RR}{L}$, and let $g \in L$ be the leading monomial of $G$; so there are nonzero $a \in \RR$ and small $\epsilon \in \Gs{\RR}{L}$  such that $G = ag(1+\epsilon)$. 
\begin{itemize}
	\item[(L1)] Assume that $a>0$.  Note that $\log \circ g$ is an $\RR$-linear combination of elements of the set $\set{\log_k:\ k \in \NN}$.  Therefore, with $F_{\log} \in \Ps{R}{X}$ the Taylor series at 0 of $\log(1+x)$, I define  $$\log \circ G:= \log a + \log \circ g + (F_{\log} \circ \epsilon).$$  Note that if $G$ is small and $G>0$, then $-\log \circ G = \log \circ \frac1G$; and if $G$ is infinitely increasing, then so is $\log \circ G$.  Thus, for $G$ infinitely increasing and nonzero $i \in \NN$, I define $$\log_i \circ G := \log \circ (\log_{i-1} \circ G)$$ by induction on $i$.
	\item[(L2)] Recall that $L' = \set{m \in L:\ m(-1) = 0}$, and let $F \in \Gs{\RR}{L'}$.  So there are $l \in \NN$ and $P \in \Gs{\RR}{(X_0, \dots, X_l)^*}$ with natural support such that $$F = P\left(\frac1{\log_0}, \dots, \frac1{\log_l}\right);$$  i.e., the support of $F$ contains no exponential monomials.  Assume  that $G$ is infinitely increasing.  Then, by (L1) above, there exist $k_i \in \NN$ and $Q_i \in \Gs{\RR}{(X_{-1}, \dots, X_{k_i})^*}$ with natural support such that $$\frac1{\log_i} \circ G = Q_i\left(\frac1\exp, \frac1{\log_0}, \dots, \frac1{\log_{k_i}}\right), \quad\text{for } i \in \NN.$$  Since $G$ is infinitely increasing, each $\frac1{\log_i} \circ G$ is small, and it follows that $P(Q_0, \dots, Q_l) \in \Gs{\RR}{(X_0, \dots, X_k)^*}$, where $k= \max\{k_0, \dots, k_l\}$.  Therefore, I set $$F \circ G:= P(Q_0, \dots, Q_l)\left(\frac1\exp, \frac1{\log_0}, \dots, \frac1{\log_k}\right) \in \Gs{\RR}{L}.$$
	\item[(L3)] Let $F \in \Gs{\RR}{L}$, and let $l \in \NN$ and $P \in \Gs{\RR}{(X_{-1}, \dots, X_l)^*}$ with natural support be such that $$F = P\left(\frac1\exp, \frac1{\log_0}, \dots, \frac1{\log_l}\right).$$  Then I set $$F \circ \log:= P\left(\frac1{\log_0}, \dots, \frac1{\log_{l+1}}\right);$$ note that $F \circ \log \in \Gs{\RR}{L'}$.
\end{itemize}

\begin{lemma}
	\label{associative_lemma}
	Let $F \in \Gs{\RR}{L'}$ and $G \in \Gs{\RR}{L}$ be such that $G$ is infinitely increasing.  Then $(F \circ \log) \circ G = F \circ (\log \circ G)$.
\end{lemma}

\begin{proof}
	Let $Q_i$ be for $\frac1{\log_i} \circ G$ be as in (L2).  Then for $i \in \NN$, I have by (L1) that $$\frac1{\log_i} \circ (\log \circ G) = \frac1{\log_{i+1}} \circ G = Q_{i+1}\left(\frac1\exp, \frac1{\log_0}, \dots, \frac1{\log_{k_{i+1}}}\right).$$  On the other hand, let $l \in \NN$ and $P \in \Gs{\RR}{(X_0, \dots, X_l)^*}$ with natural support be such that $$F = P\left(\frac1{\log_0}, \dots, \frac1{\log_l}\right).$$  Then by (L2), I have $$F \circ (\log \circ G) = P(Q_1, \dots, Q_{l+1})\left(\frac1\exp, \frac1{\log_0}, \dots, \frac1{\log_{k}}\right),$$ where $k:= \max\{k_1, \dots, k_{l+1}\}$.  On the other hand, by (L3), I have $F \circ \log = P\left(\frac1{\log_1}, \dots, \frac1{\log_{k_{l+1}}}\right)$, so again by (L2), I get $$(F \circ \log) \circ G = P(Q_1, \dots, Q_{l+1})\left(\frac1\exp, \frac1{\log_0}, \dots, \frac1{\log_{k_{i+1}}}\right),$$ and the lemma is proved.
\end{proof}

I continue working in the setting of (L1)--(L3) above.
\begin{itemize}
	\item[(L4)] For $r \in \RR$, I let $P_r \in \Ps{R}{X}$ be the Taylor series at 0 of $(1+x)^r$, and I define $$G^r := a^r g^r \cdot (P_r \circ \epsilon).$$  Note that, if $G$ is infinitely increasing, then so is $G^r$.
	\item[(L5)] For $r \in \RR$, I let $F_{\exp^r}$ be the Taylor series at 0 of the function $x \mapsto \exp(rx)$, and I set $$\exp^r \circ (\log \circ G) := a^r g^r (F_{\exp^r} \circ(F_{\log} \circ \epsilon).$$  Note that this series has order $r \cdot \ord(g)$; thus, for $F = \sum f_r \exp^{-r} \in \Gs{\RR}{L}$ with $f_r \in \Gs{\RR}{L'}$ I set $$F \circ (\log \circ G) := \sum (f_r \circ (\log \circ G)) \cdot G^{-r}.$$
\end{itemize}

\begin{cor}
	\label{associative_cor}
	Let $F, G \in \Gs{\RR}{L}$ be such that $G$ is infinitely increasing.  Then $(F \circ \log) \circ G = F \circ (\log \circ G)$.
\end{cor}

\begin{proof}
	Note that $$P_r(x) = (1+x)^r = \exp(r\log(1+x)) = (F_{\exp^r} \circ F_{\log})(x)$$ for $r \in \RR$ and small $x \in \RR$, so that $P_r \circ \epsilon = F_{\exp^r} \circ F_{\log} \circ \epsilon$.  It follows from (L3), (L4) and Lemma \ref{associative_lemma} that $F \circ (\log \circ G) = (F \circ \log) \circ G$.
\end{proof}

In the situation of the previous corollary, I write $F \circ \log \circ G$ for the composition $F \circ (\log \circ G) = (F \circ \log) \circ G$, called the \textbf{$\log$-composition} of $F$ with $G$.

\subsection*{Closure under $\log$-composition}
First I show that $\F_0$ is closed under $\log$-composition.

\begin{lemma}
	\label{first_log_comp}
	Let $f,g \in \F_0$ and assume that $g$ is infinitely increasing.  Then $f \circ \log \circ g \in \F_0$ and $T_0(f \circ \log \circ g) = T_0(f) \circ \log \circ T_0(g).$
\end{lemma}

\begin{proof}
	It suffices to prove the lemma for $f \in \A_0$.  Let $\Omega$ and $\Delta$ be strong asymptotic expansion domains for $f$ and $g$, respectively.  (Recall that ```strong asymptotic expansion'' and ``strong $\F$-asymptotic expansion'' mean the same thing for $h \in \F_0$.)  By Lemma  \ref{large_asymptotics}, after shrinking $\Omega$ if necessary, the germ $\log \circ g$ has a holomorphic extension $\ll_g$ on $\Omega$ such that $\left(\ll_g\right)(\Omega) \subseteq \Delta$.  Therefore, the function $h:= f \circ \log \circ g$ has bounded, holomorphic extension $\ff \circ \ll_g$ on $\Omega$.
	
	Moreover, for each $r \ge 0$, the germ $g^{-r} = \exp^{-r} \circ (\log \circ g)$ has bounded holomorphic extension $\bexp^{-r} \circ \ll_g$ on $\Omega$.  On the other hand, writing $g = am(1+\epsilon)$ with $a>0$, $m \in L$ the leading monomial of $g$ and $\epsilon \in \A_0$ small, I get $$g^{-r} = a^{-r} m^{-r}( P_{-r} \circ \epsilon),$$ where $P_{-r}$ is the Taylor series expansion of $x \mapsto (1+x)^{-r}$ at 0.  It follows from Lemma \ref{construction_lemma_1}(2) that $g^{-r} \in \F_0$ with strong asymptotic expansion domain $\Omega$ such that $T_0(g^{-r}) = a^{-r} m^{-r} (P_{-r} \circ T_0(\epsilon)) = T_0(g)^{-r}$ by (L1).  Setting $d:= \ord(g) < 0$, it follows in particular that $\ord(g^{-r}) = -rd$.
	
	Now say that $T_0(f) = \sum_{r \ge 0} a_r \exp^{-r}$, and let $r \ge 0$.  Since $f$ has strong asymptotic expansion $T_0(f)$ in $\Delta$, we have $$\ff - \sum_{s \le r} a_s \bexp^{-s} = o(\bexp^{-r}) \quad\text{ in } H(0)\ ,$$ so that $$\ff \circ \ll_g - \sum_{s \le r} a_s(\bexp^{-s} \circ \ll_g) = o\left(\bexp^{-r} \circ \ll_g\right) \quad\text{in } \Omega.$$  By the previous paragraph, we have $a_s g^{-s} \in \F_0$ with strong asymptotic expansion domain $\Omega$, for each $s \ge 0$, and $\ord(a_s g^{-s}) = -sd \to +\infty$ as $s \to +\infty$.  Since $T_0(f)$ has $L$-natural support, it follows from Lemma \ref{inf_sum_of_strong} that $f \in \A_0$ with $T_0(f) = \sum a_r T_0(g)^{-r}$.	 On the other hand, since $T_0(f) \circ \log = \sum a_r x^{-r}$, we have $T_0(f) \circ \log \circ T_0(g) = T_0(f)$, and the lemma is proved.
\end{proof}

Next, let $k,l \in \NN$, $f \in \F_k$ and $g \in \F_l$, and assume that $g$ is infinitely increasing.
The remaining difficulty in the proof of Theorem \ref{main_thm}(2) lies in making sense of the strong asymptotic expansion of $f \circ \log \circ g$.

\begin{nrmks}
	\label{small_rmk}
	Set $s_0:= \ord(g) \le 0$, and  let $g_{s_0} \in \F'_l$ and $\epsilon \in \A_l$ be such that $g = g_{s_0}\exp^{-s_0}(1 + \epsilon)$ and $\ord(\epsilon)>0$.  There are two cases to consider:
	\begin{itemize}
	\item[Case 1:]  $s_0<0$.  Say $\tau_k(f) = \sum f_r \exp^{-r}$ and let $r \in \supp(\tau_k(f))$.  Since $f_r \in \F'_k$, there exists $m(r) \in \NN$ such that $x^{-m(r)} \le |f_r| \le x^{m(r)}$; and since $g \in \F_l$, there exists $n(r) \in \NN$ such that $x^{-n(r)} \le \log \circ g \le x^{n(r)}$.  Hence there exists $N(r) \in \NN$ such that $$x^{-N(r)} \le f_r \circ \log \circ g \le x^{N(r)}.$$  If I already know (by induction on $k$, say) that each $f_r \circ \log \circ g$ belongs to $\F_j$ for some $j \in \NN$ independent of $r$ then, by Corollary \ref{construction_cor}(1), there exist $h_r \in \F_j'$ and $d(r) \in \RR$ such that $f_r \circ \log \circ g \sim h_r \exp^{d(r)}$.  Since (as above for $f_r$) the germ $h_r$ is also polynomially bounded, it follows that $d(r) = \ord(f_r \circ \log \circ g) = 0$, so that $$\ord\big(\tau_{j}(f_r \circ \log \circ g) \tau_l(g)^{-r}\big) = -r s_0.$$  Since $\exp^{-r} \circ \log \circ g = g^{-r}$ for each $r$, this suggests that the series $$\sum_{r \in \RR} \tau_{j}(f_r \circ \log \circ g) \tau_l(g)^{-r}$$ is a candidate for the strong asymptotic expansion of $f \circ \log \circ g$ in this case.
	\item[Case 2:] $s_0=0$.  The assumption that $g$ is infinitely increasing then implies that $g_0 \in \F'_l$ is infinitely increasing as well; in particular, we must have $l>0$.  By Taylor's Theorem, since $\log \circ g = \log \circ g_0 + F_{\log} \circ \epsilon$ and $\log \circ g_0$ is infinitely increasing while $F_{\log} \circ \epsilon$ is small, we have 
	\begin{equation*}
		f \circ \log \circ g = \sum_{i=0}^\infty \frac{f^{(i)} \circ \log \circ g_0}{i!}\ (F_{\log}\circ \epsilon)^i.
	\end{equation*}
	This suggests the following: if I already know (by induction on $l$, say) that each $f^{(i)} \circ \log \circ g_0$ belongs to $\F'_j$ for some $j \ge l$ independent of $i$, then the series $$\sum_{i=0}^\infty \frac{f^{(i)} \circ \log \circ g_0}{i!}\ \tau_l(F_{\log} \circ \epsilon)^i$$ is a candidate for the strong asymptotic expansion of $f \circ \log \circ g$ in this case.
	\end{itemize}
\end{nrmks}

In view of Case 2 above, I need a formal version of the Taylor expansion theorem.  It relies on the observation that the logarithmic generalized power series belong to the set $\TT$ of  \textit{transseries} as defined by van der Hoeven in \cite{Hoeven:2006qr}.

\begin{lemma}
	\label{formal_taylor}
	Let $F \in \Gs{\RR}{L}$, let $k>0$, and let $G \in \Gs{\RR}{L'}$ and $H \in \Gs{\RR}{L}$ be such that $G$ is infinitely increasing and $H$ is small.  Then, as elements of $\,\TT$, we have $$F \circ (G+H) = \sum_{i=0}^\infty \frac{F^{(i)}\circ G}{i!} H^i.$$  
\end{lemma}

\begin{proof}
	By \cite[Theorem 5.12]{Hoeven:2006qr}, there exists a transseries $G^{-1} \in \TT$ such that $G \circ G^{-1} = x$.  Since $H$ is small, so is the transseries $\delta:= H \circ G^{-1}$; that is, we have $\delta \prec 1$ in the notation of \cite{Hoeven:2006qr}.  On the other hand, for $m \in L$ we have $m^\dagger:= (\log m)'$ is bounded, so that $m^\dagger\delta$ is small as well.  It follows from \cite[Proposition 5.11(c)]{Hoeven:2006qr} that $$F \circ (x+\delta) = \sum_{i=0}^\infty \frac{F^{(i)}}{i!} \delta^i.$$  Composing on the right with $G$ then proves the lemma.
\end{proof}

\begin{thm}
 \label{comp-closed_1}
Let $k,l \in \NN$, $f \in \F_k$ and $g \in \F_l$, and assume that $g$ is infinitely increasing.  Then $f \circ \log \circ g \in \F_{k+l}$ and  $$T(f \circ \log \circ g) = (Tf) \circ \log \circ (Tg).$$ 
Moreover, writing $g = g_{s_0}\exp^{-s_0}(1 + \epsilon)$ with $s_0 = \ord(g)$ and $\ord(\epsilon)>0$, and writing $\tau_k(f) = \sum f_r \exp^{-r}$, we have 
\begin{equation*}
\tau_{k+l}(f \circ \log \circ g) = \begin{cases} \sum_{r \in \RR} \tau_{k-1+l}(f_r \circ \log \circ g) \tau_l(g)^{-r} &\text{if } s_0 < 0, \\ \sum_{i \in \NN} \frac{f^{(i)} \circ \log \circ g_0}{i!} \tau_l(F_{\log} \circ \epsilon)^i &\text{if } s_0 = 0, \end{cases}
\end{equation*}  
where $F_{\log}$ is the Taylor series at 0 of the function $x \mapsto \log(1+x)$.
\end{thm}

\begin{proof}
Since $\F_k$ is the fraction field of $\A_k$, I may assume that $f \in \A_k$.  By Lemma \ref{large_asymptotics} there is a strong $\F$-asymptotic expansion domain $\Omega$ of $g$ such that $\ll_g(\Omega) \subseteq \Delta$, where $\Delta$ is a strong $\F$-asymptotic expansion domain of $f$.  In particular, the germ $h:= f \circ \log \circ g$ has a holomorphic extension $\hh:= \ff \circ \ll_g$ on $\Omega$.


We proceed by induction on the pair $(k,l) \in \NN^2$ with respect to the lexicographic ordering of $\NN^2$.  The case $k=l=0$ corresponds to Lemma \ref{first_log_comp}, so I assume $(k,l)>(0,0)$ and the theorem holds for lower values of $(k,l)$.
Let $f_r \in \F'_k$ be such that $\tau_k(f) = \sum_{r \ge 0} f_r \exp^{-r}$, and let $g_r \in \F_l'$ be such that $\tau_l(g) = \sum_{r \in \RR} g_r \exp^{-r}$.  Set $s_0:= \ord(g) \le 0$; we distinguish two cases:

\subsection*{Case 1: $s_0 < 0$}  By the inductive hypothesis, each $f_r \circ \log \circ g$ belongs to $\F_{k-1+l}$.  Since $f_r \in \RR$ if $k=0$ and $\F'_l \subseteq \F'_{k-1+l}$ if $k>0$, it follows from Remark \ref{small_rmk}(1) that the series $$H:= \sum_{r \ge 0} \tau_{k-1+l}(f_r \circ \log \circ g) \tau_l(g)^{-r}$$ belongs to $\Gs{\F'_{k-1+l}}{E} \subseteq \Gs{\F'_{k+l}}{E}$, and I claim that $\tau_{k+l}(h) = H$.

To prove the claim, let $r \in \supp(\tau_k(f))$; it suffices, by Lemma \ref{inf_sum_of_strong}, to show that 
\begin{equation*}
	\hh - \sum_{s \le r} (\ff_s \circ \ll_g) \gg^{-s} = o\left((\ff_r \circ \ll_g) \gg^{-r} \right) \quad\text{in } \Omega.
\end{equation*}
However, by assumption I have $\ff - \sum_{s \le r} \ff_s \bexp^{-s} = o\left(\bexp^{-r'}\right)$ in $\Delta$, for any $r'>r$ such that $r' < \ord\left(f - \sum_{s \le r} f_s \exp^{-s}\right)$; in particular, $$\hh - \sum_{s \le r} (\ff_s \circ \ll_g) \gg^{-s} = o\left(\gg^{-r'}\right) \quad\text{in } \Omega.$$  On the other hand, by Case 1 of Remark \ref{small_rmk}, the germ $f_r \circ \log \circ g$ is polynomially bounded, so that $\gg^{-r'} = o\left((\ff_r \circ \ll_g) \gg^{-r}\right)$ in $\Omega$, which proves the claim.

Finally, by the inductive hypothesis I have, for $r \ge 0$, that
\begin{equation*}
T\left(\sum_{s \le r} \frac{f_s \circ \log \circ g}{g^{s}}\right) = \sum_{s \le r} \frac{T(f_s) \circ \log \circ T(g)}{T(g)^{s}} = (T(f))_r \circ \log \circ T(g).
\end{equation*}
Since $\ord\left((f_s \circ \log \circ g)g^{-s}\right) \to +\infty$ as $s \to +\infty$, we get $T(h) = T(f) \circ \log \circ T(g)$, and the theorem is proved in this case.

\subsection*{Case 2: $s_0 = 0$}  Then $l>0$ and there exists $h_0 \in \F_{l-1}$ such that $g_0 = h_0 \circ \log$. By the inductive hypothesis and  Proposition \ref{derivative_prop}, each $f^{(i)} \circ \log \circ h_0$ belongs to $\F_{k+l-1}$, so that $f^{(i)} \circ \log \circ g_0$ belongs to $\F'_{k+l}$; in particular, the series $$H:= \sum_{i \in \NN} \frac{f^{(i)} \circ \log \circ g_0}{i!} \tau_l(F_{\log} \circ \epsilon)^i$$ belongs to $\Gs{\F'_{k+l}}{E}$, where $\epsilon:= (g - g_0)/g_0$.  Based on Case 2 of Remark \ref{small_rmk}, I now claim that $\tau_{k+l}(h) = H.$  
	
To prove the claim, note first that it is clear from Case 2 of Remark \ref{small_rmk} if $f^{(n)} = 0$ for some $n \in \NN$, since the series $H$ is given by a finite sum in this case.  So assume from now on $f^{(n)} \ne 0$ for all $n$; since $\ord(F_{\log} \circ \epsilon)>0$, we have $$\ord\left((F_{\log} \circ \epsilon)^i\right) \to \infty \text{ as } i \to \infty.$$  Shrinking $\Omega$ if necessary, we may assume that $\Omega$ is also a strong $\F$-asymptotic expansion domain of $\epsilon$ and of $\log \circ g_0$, with corresponding holomorphic extensions $\ee$ and $\ll_{g_0}$, respectively.  By Lemma \ref{inf_sum_of_strong}, it therefore suffices to show that  $$\hh - \sum_{i=0}^n \frac{\ff^{(i)} \circ \ll_{g_0}}{i!} (F_{\log} \circ \ee)^i = o\left(\frac{\ff^{(n)} \circ \ll_{g_0}}{n!} (F_{\log} \circ \ee)^n\right)$$ in $\Omega$, for $n \in \NN$.  However, it follows from Corollary \ref{construction_cor}(1) that $\left|\ff^{(n+1)}(z)\right| \le e^{p|z|}$ for some $p \in \NN$ and sufficiently large $z \in \Omega$.  Also, since $T(g_0) \in \F_l'$ and $g_0$ is infinitely increasing, the leading monomial of $g_0$ belongs to $L'$, so the leading monomial of $\log \circ g_0$ is $\log_i$ for some $i \ge 1$; hence $\left|\ll_{g_0}(z)\right| \le q \log |z|$ for some $q \in \NN$ and sufficiently large $z \in \Omega$.  Finally, since $\ord(\epsilon)>0$, it follows that $|(F_{\log} \circ \ee)(z)| \le |z|^r|e^{-sz}|$ for sufficiently large $z \in \Omega$, where $s = \ord(F_{\log} \circ \epsilon)>0$ and $r \in \NN$.  Combining these three estimates with Taylor's formula, one obtaines
\begin{equation*}
	\left|\hh - \sum_{i=0}^n \frac{\ff^{(i)} \circ \ll_{g_0}}{i!} (F_{\log} \circ \ee)^i \right| \le K \left|\xx^t \bexp^{-(n+1)s}\right|
\end{equation*}	
in $\Omega$, for some $t \in \NN$ and $K>0$.  On the other hand, since $\left|\ff^{(n)}(z)\right| \ge e^{-p|z|}$ for some $p \in \NN$ and sufficiently large $z \in \Omega$, since $\left|\ll_{g_0}(z)\right| \le q\log |z|$ for some $q \in \NN$ and sufficiently large $z \in \Omega$, and since $|(F_{\log} \circ \ee)(z)| \ge |z|^{-r}|e^{-sz}|$ for sufficiently large $z \in \Omega$ for some $r \in \NN$, we have $$\left| \frac{\ff^{(n)} \circ \ll_{g_0}}{n!} (F_{\log} \circ \ee)^n \right| \ge K'\left|\xx^{-u} \bexp^{-ns}\right|$$ in $\Omega$, for some $u \in \NN$ and $K'>0$.  By Lemma \ref{strong_order_cor}, we have  $$\xx^t\bexp^{-(n+1)s} = o(\xx^{-u}\bexp^{-ns}) \quad\text{in } \Omega,$$ so the claim follows.

Finally, since $\ord(F_{\log} \circ \epsilon)^i \to \infty$ as $i \to \infty$, it follows from the inductive hypothesis, Proposition \ref{derivative_prop} and Lemma \ref{formal_taylor} that 
\begin{align*}
	T(h) &= \sigma(\tau_k(h)) \\ &= \sum_{i \in \NN} \frac{T\left(f^{(i)} \circ \log \circ g_0\right)}{i!} T(F_{\log} \circ \epsilon)^i \\ &= \sum_{i \in \NN} \frac{T(f)^{(i)} \circ \log \circ T(g_0)}{i!} F_{\log} \circ T(\epsilon)^i \\ &= T(f) \circ \left(\log \circ T(g_0) + F_{\log} \circ T(\epsilon)\right) \\ &= T(f) \circ \log \circ T(g),
\end{align*} 
so the theorem follows in this case as well. 
\end{proof}

\section{Concluding remarks}\label{conclusion_section}

As mentioned in the introduction, the purpose of this paper is to extend Ilyashenko's construction in \cite{Ilyashenko:1991fk} of the class of almost regular maps to obtain a qaa field containing them.  My reason for doing so is the conjecture that this class generates an o-minimal structure over the field of real numbers.  This conjecture, in turn, might lead to locally uniform bounds on the number of limit cycles in subanalytic families of real analytic planar vector fields all of whose singularities are hyperbolic; see \cite{Kaiser:2009ud} for explanations and a positive answer in the special case where all singularities are, in addition, non-resonant.  (For a different treatment of the general hyperbolic case, see Mourtada \cite{mourtada2009}.) 

My hope is to settle the general hyperbolic case by adapting the procedure in \cite{Kaiser:2009ud}, which requires three main steps:
\begin{enumerate}
	\item extend Ilyashenko's class $\A$ into a qaa algebra;
	\item construct such algebras in several variables, such that the corresponding system of algebras is stable under various operations (such as blowings-up, say);
	\item obtain o-minimality using a normalization procedure.
\end{enumerate}
While this paper contains a first successful attempt at Step (1), Step (2) poses some challenges.  For instance, it is not immediately obvious what the nature of logarithmic generalized power series in several variables should be; they should at least be stable under all the operations required for Step (3).  

In collaboration with Tobias Kaiser, I am taking the approach of enlarging the set of monomials itself, in such a way that this set is already stable under the required operations; a natural candidate for such a set of monomials is the set of all functions definable in the \hbox{o-minimal} structure $\Ranexp$ (see van den Dries and Miller \cite{Dries:1994eq} and van den Dries et al. \cite{Dries:1994tw}).  However, working with this large set of monomials requires us to revisit Step (1) and further adapt the construction discussed here to the corresponding generalized power series.  A joint paper (in collaboration with Tobias Kaiser and my student Zeinab Galal) addressing this generalization of Step (1) is in preparation.


\end{document}